\documentclass[11pt,twoside]{article}
\usepackage{amsthm}
\usepackage{amsmath}
\usepackage{amssymb}
\usepackage{array}
\usepackage{amsfonts,amscd}
\usepackage{exscale}
\usepackage{eucal}
\usepackage{latexsym,amsmath}
\usepackage{indentfirst}
\newcommand{\e}{\varepsilon}
\newcommand{\w}{\widetilde}

\numberwithin{equation}{section}
\newtheorem{Prop}{\bf Proposition}[section]
\newtheorem{Cor}{\bf Corollary}[section]
\newtheorem{defn}{\bf Definition}[section]
\newtheorem{Rem}{\bf Remark}[section]

\newtheorem{Th}{Theorem}[section]

\begin{document}
\def \b{\Box}

\begin{center}
{\Large {\bf A GROUPOID STRUCTURE ON A VECTOR SPACE}}
\end{center}

\begin{center}
{\bf Vasile POPU\c TA and Gheorghe IVAN}
\end{center}

\setcounter{page}{1}

\pagestyle{myheadings}

{\small {\bf Abstract}. In this paper we introduce the concept of
generalized vector groupoid. Several properties of them are
established.}
{}

\section{Introduction}
\indent\indent The notion of  groupoid was introduced by H. Brandt
[Math. Ann., \textbf{96}(1926), 360-366; MR 1512323]. This
algebraic structure is similar to a group, with the exception that
products of elements cannot are always be defined.

A generalization of Brandt groupoid has appeared  in a paper of C.
Ehresmann [{\it Oeuvres completes. Parties I.1, I.2}. Dunod,
Paris,1950]. Groupoids and their generalizations (topological
groupoids, Lie and symplectic groupoids etc.) are mathematical
structures that have proved to be useful in many areas of science
(see for instance \cite{cone}, \cite{cdwe}, \cite{gost},
\cite{mack}, \cite{rare}, \cite{wein}).

The concept of vector groupoid has been defined by V. Popu\c ta
and Gh. Ivan \cite{poiv}. In this paper we introduce a groupoid
structure in the sense of Ehresmann on a vector space.

 The paper is organized as follows. In Section 2  we present some basic facts
 about  Ehresmann groupoids. In Section 3  we introduce the notion of generalized  vector
groupoids and its useful properties are established.  The
construction of the induced vector groupoid and some
characterizations of them are given in Section 4.

\section{ Ehresmann groupoids}
\indent\indent We recall the minimal necessary backgrounds on
Ehresmann groupoids (for further details see e.g. \cite{mack},
\cite{wein}).

\begin{defn}(\cite{mack}) A \textbf{groupoid $ G $
over} $ G_{0} $ (\textbf{in the sense of Ehresmann}) is a pair $
(G, G_{0})$ of sets endowed with two surjective maps $\alpha,\beta
:G \rightarrow G_{0}$ (called {\bf source}, respectively {\bf
target}), a partially binary operation (called {\bf
multiplication}) $~m :G_{(2)}:=\{( x,y)\in G\times G~|~\beta(
x)=\alpha(y)\}\rightarrow G,~(x,y)\to m(x,y):=x\cdot y,~$ ($G
_{(2)}$ is the {\bf set of composable pairs}), an injective map $
\e:~G_{0} \longrightarrow ~G~$ (called {\bf inclusion map}) and a
map $~i :G \rightarrow G,~x\to~ i(x):=x^{-1}$ (called
\textbf{inversion}), which verify the following conditions:

(G1)~({\bf associativity}): $~(x\cdot y)\cdot z=x\cdot(y\cdot z)$
in the sense that if one of two products $(x\cdot y)\cdot z$ and
$x\cdot(y\cdot z)$ is defined, then the other product is also
defined and they are equals;

(G2)~({\bf units}): for each  $~x\in~G~\Longrightarrow
(\e(\alpha(x)),x)\in G_{(2)}, (x,\e(\beta(x)))\in ~G_{(2)}~$ and
$~\e(\alpha(x))\cdot x = x = x\cdot \e(\beta(x))$;

 (G3)~({\bf inverses}): for each $~x\in G~\Longrightarrow~
(x^{-1},x)\in G_{(2)},~ (x,x^{-1}) \in G_{(2)}~$ and $ ~x^{-1} x =
\e(\beta(x)),~ x x^{-1}=\e(\alpha(x)).$
\end{defn}

\markboth{Vasile Popu\c ta, Gheorghe Ivan}{ A groupoid structure
on a vector space}

For a groupoid $~G~$ we sometimes write $~(G, \alpha, \beta, m,
\e, \iota, G_{0})~$ or $~(G, \alpha, \beta, G_{0})~$ or $~(G,
G_{0})~$; $~G_{0}~$ is called the {\it base} of $~G~$. The
functions $~\alpha, \beta, m, \e, i~$ are called the {\it
structure functions} of $G$.
 The element $\varepsilon(\alpha(x))$ respectively $\varepsilon(\beta(x))$ is called the
\textit{left unit} respectively \textit{right unit} of $x;$
$~\varepsilon(G_{0})$ is called the \textit{unit set} of $G$. For
each $ u\in G_{0} $, the set $ \alpha^{-1}(u)$ (resp. $
\beta^{-1}(u)$ ) is called \textit{$\alpha-$ fibre} (resp.
\textit{$\beta-$ fibre}) of $ G$ at $ u\in G_{0}$.

{\bf Convention.} We write sometimes $~x y~$ for $~m(x,y)~$, if $
(x,y) \in G_{(2)}.$ Whenever we write a product in a given
groupoid, we are assuming that it is defined.\hfill$\b$

If $~(G, \alpha, \beta ; G_{0})~$ is a groupoid, the map $~(
\alpha, \beta):~G~\to~ G_{0}\times G_{0}~$ defined by $~(\alpha ,
\beta)(x):= (\alpha(x), \beta(x)),~(\forall)~ x\in G~$ is called
the {\it anchor map } of $~G.~$ If the anchor map $~(\alpha,
\beta) : G~\to~G_{0}\times G_{0}~$ is surjective, we say that $(G,
G_{0})$ is {\it transitive}.

In the following proposition we summarize some basic rules of
algebraic calculation in a Ehresmann groupoid obtained directly
from definitions.

\begin{Prop}(\cite{ivan}) In a Ehresmann groupoid $~( G, \alpha, \beta, m,\e, i, G_{0})~$ the
following assertions hold :

$(i)~~~\alpha(x y) = \alpha(x)~$ and $~\beta(x y)=\beta(y)~$ for
any $~(x,y)\in G_{(2)};$

$(ii)~~\alpha( x^{-1}) =\beta (x)~~\hbox{and}~~\beta(x^{-1})
=\alpha(x),~ \forall x\in G;$

$(iii)~\alpha(\varepsilon(u))= u ~~\hbox{and}~~ \beta
(\varepsilon(u)) = u ,~~ (\forall )~ u\in G_{0};$

$(iv)~~\varepsilon(u)\cdot\varepsilon(u)=\varepsilon (u)~~
\hbox{and}~~(\varepsilon(u))^{-1}=\varepsilon(u)~$ for each $~
u\in G_{0};$

$(v)~~~$ if $(x,y)\in G_{(2)},$ then $(y^{-1},x^{-1})\in G_{(2)}$
and  $~(x\cdot y)^{-1}=y^{-1}\cdot x^{-1};$

$(vi)~~$ for any $~u\in G_{0},~$ the set $~G(u):=
\alpha^{-1}(u)\cap \beta^{-1}(u)~$ is a group under the
restriction of the multiplication, called the {\bf isotropy group
at} $~u~$ of $~G;~$

$(vii)~~\varphi : G(\alpha(x)) \to G(\beta(x)),~\varphi (z):=
x^{-1} z x $ is an isomorphism of groups.

$(viii)~$ if $ (G, G_{0} ) $ is transitive, then all isotropy
groups are isomorphic.
 \end{Prop}

Applying Proposition 2.1, it is easily to prove the following
proposition.

\begin{Prop}(\cite{ivan})
Let $~( G, \alpha, \beta, m, \varepsilon, i, G_{0})~$ be a
Ehresmann groupoid. The structure functions of $G$ verifies the
following relations:

$(i)~~~\alpha \circ i = \beta,~~ \beta \circ i = \alpha, ~~i\circ
\varepsilon = \varepsilon ~~ \hbox{and}~~ i \circ i = Id_{G};$

$(ii)~~ \alpha \circ \varepsilon= \beta \circ \varepsilon=
Id_{G_{0}}.$
\end{Prop}

\begin{Rem}
Let $(G, \alpha, \beta, m, \varepsilon, i, G_{0} )$ be a Ehresmann
groupoid. If $G_{0} \subseteq G$ and $\varepsilon = Id_{G_{0}} $,
then $(G, \alpha, \beta, m, i, G_{0} )$ is a Brandt groupoid
(\cite{poiv}).
\end{Rem}
\begin{defn}(\cite{mack})
Let $~(G, \alpha, \beta, G_{0})~$ and $~(G^{\prime},
\alpha^{\prime}, \beta^{\prime}, G_{0}^{\prime})~$ be two
groupoids.

$(i)~$ A {\bf morphism of groupoids} or {\bf groupoid morphism}
from $~G~$ into $~G^{\prime}~$ is a pair of maps $~( f, f_{0} ),~$
where $~f: G \to G^{\prime}~$ and $~f_{0} :
G_{0}~\to~G_{0}^{\prime}~$ such that the following conditions
hold:

$(1)~~~\alpha^{\prime}\circ f = f_{0} \circ
\alpha,~\beta^{\prime}\circ f = f_{0} \circ \beta $;

$(2)~~~f(m(x,y)) = m^{\prime}(f(x),f(y))~$ for all $~(x,y)\in
G_{(2)}.$

$(ii)~$ A groupoid morphism $~( f, f_{0} ): (G,
G_{0})~\to~(G^{\prime}, G_{0}^{\prime})$ such that $f$ and $f_{0}$
are bijective maps, is called {\bf isomorphism of groupoids}.
\end{defn}

\begin{Prop}(\cite{ivan})
If $~( f, f_{0} ): (G, G_{0})~\to~(G^{\prime}, G_{0}^{\prime})$ is
groupoid morphism, then the following relations hold:\\[-0.5cm]
$$f \circ \varepsilon = \varepsilon^{\prime}  \circ f_{0}~~~\hbox{and}~~~ f
\circ i = i^{\prime}  \circ f.$$\\[-1.21cm]
\end{Prop}

If $~G_{0} = G_{0}^{\prime}~$ and $~f_{0} = Id_{G_{0}},~$ we say
that $~f : ( G , G_{0})~\to~( G^{\prime}, G_{0} )~$ is a
 {\it $~G_{0}-$  morphism of groupoids}.

{\bf Example 2.1.}$(i)~$ A nonempty set $~G_{0}~$ may be
considered to be a groupoid over $~G_{0},~$ called the {\it nul
groupoid} associated to $~G_{0}.~$ For this, we take $~\alpha =
\beta = \varepsilon = i = Id_{G_{0}}~$ and $~ u\cdot u = u~$ for
all $~u\in G_{0}.~$

$(ii)~$ A group $~G~$ having $~e~$ as unity may be considered to
be a $~\{ e \}$ - groupoid with respect to structure
functions:\\[0.2cm]
$\alpha (x) = \beta (x): = e$; $~~G_{(2)}:=\{ {(x,y) \in G\times
G~|~\beta (x) =\alpha (y)}\}= G\times G,~~ m (x,y): =
xy$;\\[0.2cm]
$\varepsilon:\{ e \} \to G,~~\varepsilon(e):= e~$ and $~i:G \to G,
~i(x):= x^{ - 1}$.

Conversely, a groupoid with one unit, i.e. $G_{0}=\{e\}$, is a
group.

$(iii)~$ The Cartesian product $~G:= X \times X~$ has a structure
of groupoid over  $ X $  by taking the structure functions as
follows: $~\overline{\alpha}(x,y):= x,~ \overline{\beta}(x,y):=
y;~$ the elements $~ (x,y)~$ and $~(y^{\prime},z)~$ are composable
in $~G:=X \times X~$ iff $~y^{\prime} = y~$ and we define
$~(x,y)\cdot (y,z) = (x,z)~$, the inclusion map $
\overline{\varepsilon}:X \to X\times X$ is given by
$\overline{\varepsilon}(x):=(x,x)$ and the inverse of $~(x,y)~$ is
defined by $~(x,y)^{-1}:= (y,x).~$ This is called the {\it pair
groupoid} associated to set $X$. Its unit set is $~G_{0}=\{
(x,x)\in X\times X | x\in X\}.$ \hfill$\b$

\section{Generalized vector groupoids}

 In this section we introduce a structure of Ehresmann groupoid  on a pair of vector spaces.
\begin{defn}
A {\bf generalized vector groupoid over a field $K$}, is a
Ehresmann groupoid $~(V, \alpha, \beta, \odot, \varepsilon, i,
V_0)$ such that:

\noindent(3.1.1) $~V$ and $V_{0}$ are vector spaces over $K$.

\noindent(3.1.2) The source and the target maps $\alpha : V \to
V_{0}$ and $\beta : V \to V_{0}$ are linear maps.

\noindent(3.1.3) The inclusion $ \varepsilon : V_{0} \to V$
 and the inversion  $~i:V\longrightarrow V,\ x\longmapsto
i(x):=x^{-1}$ are linear maps and the
following condition is verified:\\[0.2cm]
$(1)~~~~~~~~~~~
x+x^{-1}=\varepsilon(\alpha(x))+\varepsilon(\beta(x)),\ \mbox{for
all } x\in V.$\\[0.2cm]
(3.1.4) The map $~m:V_{(2)}:=\{(x,y)\in V\times
V~|~\alpha(y)=\beta(x)\} \to V,$ $(x,y)\longmapsto m(x,y):=x\odot
y,~$  satisfy the following conditions :
\begin{enumerate}
  \item $x\odot(y+z-\varepsilon(\beta(x)))=x\odot y+x\odot z-x$, for all $x,y,z\in V$, such that $\alpha(y)=\beta(x)=\alpha(z)$.
  \item $x\odot(ky+(1-k)\varepsilon(\beta(x)))=k(x\odot y)+(1-k)x$, for all $x,y\in V$, such that $\alpha(y)=\beta(x)$.
  \item $(y+z-\varepsilon(\alpha(x)))\odot x=y\odot x+z\odot x-x$, for all $x,y,z\in V$, such that $\alpha(x)=\beta(y)=\beta(z)$.
  \item $(ky+(1-k)\varepsilon(\alpha(x)))\odot x=k(y\odot x)+(1-k)x$ for all $x,y\in V$, such that $\alpha(x)=\beta(y)$.
\end{enumerate}
\end{defn}
When there can be no confusion we put $xy$ or $x\cdot y$ instead
of $x\odot y$.

From Definition 3.1 follows the following corollary.
\begin{Cor}
Let $~(V, \alpha, \beta, \odot, \iota,  V_0)$ be a vector
groupoid. Then:

$(i)~~~$ The source and target $ \alpha, \beta : V \to V_{0} $ are
linear epimorphisms.

$(ii)~~$ The inversion $ \iota : V \to V $ is a linear
automorphism.

$(iii)~$ The fibres $\alpha^{-1}(0)$ and $ \beta^{-1}(0) $ and the
isotropy group\\ $ V(0):= \alpha^{-1}(0)\cap \beta^{-1}(0) $ are
vector subspaces of the vector space $V$.
\end{Cor}
\begin{Prop}
If  $(V, \alpha, \beta, \odot,\varepsilon, i, V_0)$ is a vector
groupoid, then:

$(i)~~~~ \varepsilon(0)\odot x=x,\ \forall\ x\in \alpha^{-1}(0)$;

$(ii)~~~ x\odot \varepsilon(0)=x,\ \forall\ x\in \beta^{-1}(0)$.
\end{Prop}
\begin{proof}
(i) If $x\in\alpha^{-1}(0)$, then $\alpha(x)=0$. We have
$\beta(\varepsilon(0))=0$, since $\beta\circ \varepsilon =
Id_{V_{0}}$. So $(\varepsilon(0),x)\in V_{(2)}$ and, using the
condition {\it (G2)} from Definition 2.1, one obtains that
$\varepsilon(0)\cdot x=\varepsilon(\alpha(x))\cdot x = x.~$
Similarly, we prove that the assertion(ii) holds.
\end{proof}
\begin{Rem}
If in Definition 3.1, we consider $ V_{0}\subseteq V$ and $
\varepsilon = Id_{V_{0}},$ then\\ $(V, \alpha, \beta, \odot,
\varepsilon, i, V_{0})$ is a vector groupoid, see \cite{poiv}. In
this case, we will say that $(V, V_{0})$ is a vector $V_{0}-$
groupoid.
\end{Rem}
{\bf Example 3.1.} Let $V$ be a vector space over a field $K$. If
we define the maps $ \alpha_{0}, \beta_{0}, \iota_{0}: V\to
V,~\alpha_{0}(x)=\beta_{0}(x)=0,~ \iota_{0}(x)=-x,$ and $ m
_{0}(x,y)=x+y$, then $(V, \alpha_{0}, \beta_{0}, m_{0}, \iota_{0},
V_{0}=\{0\})$ is a vector groupoid called {\it vector groupoid
with a single unit}.
 Therefore, each vector space $ V $ can be regarded as vector $\{0\}-$
 groupoid.\hfill$\b$

{\bf Example 3.2.} Let $V$ be a vector space over a field $K$.
Then $ V$ has a structure of null groupoid over $V$ (see Example
2.1(i)). The structure functions are $~\alpha = \beta
=\varepsilon=\iota = Id_{V}$  and $ x\odot x = x $ for all $x\in
V$. We have that $V_{0} = V$ and the maps $ \alpha,
\beta,\varepsilon, \iota $ are linear. Since $ x+\iota(x) = x+ x $
and $ \alpha(x) + \beta(x) = x+x $ imply that the condition
3.1.3(1) holds. It is easy to verify the conditions 3.1.4(1)-
3.1.4(4) from Definition 3.1. Then $V$ is a vector groupoid,
called the {\it null vector groupoid} associated to $V$.\hfill$\b$

 {\bf Example 3.3.} Let $V$ be a vector space over
a field $K$. We consider the pair groupoid $( V\times V,
\overline{\alpha}, \overline{\beta},
\overline{m},\overline{\varepsilon}, \overline{i}, V )$ associated
to $V$ (see Example 2.1(iii)). We have that $ V\times V$ is a
vector space and the source $\overline{\alpha}$ and target
$\overline{\beta}$ are linear maps. Also, the inclusion map
$\overline{\varepsilon}: V \to V\times V$ and the inversion map
$\overline{i} : V\times V \to V\times V$ are
linear.  For all $(x,y)\in V\times V$ we have\\
$(x,y)+\overline{i}(x,y)=(x,y)+(y,x)=(x+y,x+y)$ and\\
$\overline{\varepsilon}(\overline{\alpha}(x,y))+
\overline{\varepsilon}(\overline{\beta}(x,y))=\overline{\varepsilon}(x)+
\overline{\varepsilon}(y)=(x,x)+(y,y)=(x+y,x+y)$\\
and it follows that
$~(x,y)+\overline{i}(x,y)=\overline{\varepsilon}(\overline{\alpha}(x,y))+
\overline{\varepsilon}(\overline{\beta}(x,y)).$

Therefore, the conditions $(3.1.1)- (3.1.3) $ are satisfied.

For to prove that the condition (3.1.4)(1) is verified, we
consider the elements $ x=(x_{1}, x_{2}), y=(y_{1}, y_{2}),
z=(z_{1}, z_{2}) $ from $V\times V$ such that
$\overline{\alpha}(y)=\overline{\beta}(x)=\overline{\alpha}(z).$
Then $ y_{1}=x_{2}=z_{1}$. We have\\
$ x\odot_{V\times V} (y +z
-\overline{\varepsilon}(\overline{\beta}(x)))=(x_{1},
x_{2})\odot_{V\times V} ( (x_{2}, y_{2}) + (x_{2}, z_{2})
-\overline{\varepsilon}(\overline{\beta}(x_{1}, x_{2})))=$\\
$=(x_{1}, x_{2})\odot_{V\times V} ( (x_{2}, y_{2}) + (x_{2},
z_{2}) -\overline{\varepsilon}( x_{2}))= (x_{1},
x_{2})\odot_{V\times V} ( (x_{2}, y_{2}) + (x_{2}, z_{2}) -$\\
$-( x_{2}, x_{2}))= (x_{1}, x_{2})\odot_{V\times V}(x_{2}, y_{2} +
z_{2} - x_{2})=(x_{1},y_{2} + z_{2} - x_{2}).$

On the other hand, we have\\
$ x\odot_{V\times V} y +x\odot_{V\times V} z - x = (x_{1},
x_{2})\odot_{V\times V} (x_{2}, y_{2}) + (x_{1},
x_{2})\odot_{V\times V} (x_{2}, z_{2})- (x_{1}, x_{2})= (x_{1},
y_{2})+ (x_{1}, z_{2})- (x_{1}, x_{2})=(x_{1}, y_{2}+ z_{2}-
x_{2}).$

Hence, $ x\odot_{V\times V} (y +z
-\overline{\varepsilon}(\overline{\beta}(x)))= x\odot_{V\times V}
y +x\odot_{V\times V} z - x $ and so the relation (3.1.4)(1)
holds.

 In the same manner, we prove that the relations 3.1.4(2) -
3.1.4(4) are verified. Hence $V\times V$ is a vector groupoid
called the  {\it pair vector groupoid} associated to
$V$.\hfill$\b$

\begin{defn}
Let $ ( V_{1}, \alpha_{1}, \beta_{1}, \odot_{1}, \varepsilon_{1},
i_{1}, V_{1,0} )$ and  $ ( V_{2}, \alpha_{2}, \beta_{2},
\odot_{2}, \varepsilon_{2}, i_{2}, V_{2,0} )$  be two vector
groupoids. A groupoid morphism  $ (f, f_{0}): (V_{1}, V_{1,0})\to
(V_{2}, V_{2,0}) $ with property that $ f: V_{1} \to V_{2} $ and $
f_{0}: V_{1,0} \to V_{2,0}$ are linear maps, is called  {\bf
vector groupoid morphism} or {\bf morphism of vector groupoids}.

If $ V_{2,0}=V_{1,0} $ and $ f_{0}=Id_{V_{1,0}}$, then we say that
$ (f, Id_{V_{1,0}}): (V_{1}, V_{1,0})\to (V_{2}, V_{1,0}) $ is a
{\bf $ V_{1,0}-$ morphism of vector groupoids}.  It is denoted  by
$ f: V_{1} \to V_{2}$.
\end{defn}
\begin{Prop}
Let $(V,\alpha, \beta, \odot, \varepsilon, i , V_{0})$ be a vector
groupoid. The anchor map $(\alpha, \beta): V \to V_{0}\times V_{0}
$ is a $ V_{0}-$ morphism of vector groupoids from the vector
groupoid $(V, V_{0})$  into the pair vector groupoid $ (
V_{0}\times V_{0}, \overline{\alpha}, \overline{\beta},
\overline{m}, \overline{\varepsilon}, \overline{i}, V_{0}) $.
\end{Prop}
\begin{proof}
 We denote $ (\alpha, \beta):=f $. Then $ f(x)=(\alpha(x),\beta(x)),$ for all $x\in V.$ We prove that $
 \overline{\alpha}\circ f= Id_{V_{0}}\circ \alpha $ or
 equivalently, $
 \overline{\alpha}\circ f = \alpha $. Indeed, for all $x\in V$ we have
 $
 (\overline{\alpha}\circ f)(x)=\overline{\alpha}(f(x))=\overline{\alpha}(\alpha(x), \beta(x))
 =\alpha(x).$

 Therefore, $
 \overline{\alpha}\circ f = \alpha $. Similarly, we verify that $
 \overline{\beta}\circ f = \beta $. Hence, the condition (i) from
 Definition 2.2 is satisfied. For $(x,y)\in V_{(2)}$ we have\\
 $ f(x\odot_{V} y )= (\alpha(x\odot_{V}y),
\beta(x\odot_{V}y))= (\alpha(x), \beta(y)) ~$ and\\
 $\w{m}(f(x),f(y))=
\w{m}( (\alpha(x),\beta(x)), (\alpha(y),\beta(y))) =
(\alpha(x),\beta(y)).$

Therefore, $ f( x\odot_{V} y )=\w{m}(f(x),f(y))$, for all
$(x,y)\in V_{(2)}$. Hence, the equality (ii) from Definition 2.2
is verified. Thus $ f : V \to V_{0}\times V_{0} $ is a  $V_{0}-$
morphism of groupoids.

Let $ x, y \in V $ and $ a,b\in K$. Since $ \alpha, \beta $ are linear maps, we have\\
$ f(ax+by) = (\alpha(ax+by), \beta(ax+by)) = ( a \alpha(x) +
b\alpha(y), a \beta(x) + b\beta(y))= $\\
$= a(\alpha(x) , \beta(x)) + b (\alpha(y) , \beta(y))= a f(x) + b
f(y) $,\\
i.e. $f$ is a linear map. Therefore, the conditions from
Definition 3.2 are verified. Hence $f$ is a $ V_{0}-$ morphism of
vector groupoids.
\end{proof}

\section{ The induced vector groupoid of a vector groupoid by a linear map}

Let $(V, \alpha, \beta, \odot_{V}, \varepsilon, i, V_{0})$ be a
vector groupoid (in the sense of Definition 3.1) and let
$~h:X\longrightarrow V_{0}~$ be a linear map. We consider the
set:\\[-0.3cm]
$$h^*(V)=\{(x,y,a)\in X\times X\times V~|~h(x)=\alpha(a),~
h(y)=\beta (a)\}.$$\\[-0.9cm]

Since $X$ and $V$ are vector spaces and $h$ is a linear map, on
obtains that $ h^{*}(V) $ has a canonical structure of vector
space.

For the pair $(h^{*}(V),X)$ we define the following structure
functions.

The source and target $\alpha^{*}, \beta^{*}:h^{*}(V) \to X $,
inversion $ i^{*}:h^{*}(V)\to h^{*}(V)~$ and inclusion  $
\varepsilon^{*}: X \to h^{*}(V)~$ are defined by\\[-0.2cm]

$\alpha^{*}(x,y,a):=x ,~~
\beta^{*}(x,y,a):=y,~~i^{*}(x,y,a):=(y,x,i(a)),~(\forall)~(x,y,a)\in
h^{*}(V)$,\\[-0.2cm]

$ \varepsilon^{*}(x):=(x,x,\varepsilon(h(x))), ~~(\forall)~ x\in
X$.\\[-0.2cm]

 The partially multiplication  $\odot_{h^{*}(V)}: h^{*}(V)_{(2)} \to h^{*}(V)
$, where\\
 $h^{*}(V)_{(2)}:=\{((x,y,a),(y^{\prime},z,b))\in
h^{*}(V)\times h^{*}(V)~|~y=y^{\prime}~ \hbox{and}~(a,b)\in
V_{(2)}\}$ is  given by:\\[-0.6cm]
 $$(x,y,a)\odot_{h^{*}(V)}(y,z,b):=(x,z, a\odot_{V}
 b).$$\\[-1.1cm]
\begin{Prop}
$(h^{*}(V), \alpha^{*}, \beta^{*},
\odot_{h^{*}(V)},\varepsilon^{*}, i^{*}, X )$ is a vector
groupoid, called the {\bf induced vector groupoid
 of $(V,V_{0})$ by the linear map $ h: X \to V_{0}$}.
\end{Prop}
\begin{proof}
$(i)~$  It must to verify that the conditions of Definition 2.1
are satisfied. For this, we consider an arbitrary element $
x=(x_{1}, y_{1},a_{1})\in h^{*}(V)$. Then $ h(x_{1})=\alpha(a_{1})
$ and $ h(y_{1})=\beta(a_{1})$.

$(a)~$ Let now  $y,z\in h^{*}(V)$ such that $
\beta^{*}(x)=\alpha^{*}(y) $ and $ \beta^{*}(y)=\alpha^{*}(z) $.
Then, if we take $ y=(x_{2}, y_{2},a_{2}), z=(x_{3},
y_{3},a_{3})$, it follows $ y_{1}=x_{2} $ and $ y_{2}=x_{3} $.
Therefore, $  y=(y_{1}, y_{2},a_{2}) $ and  $z=(y_{2}, y_{3},a_{3})$. We have\\[0.2cm]
$(1)~~~ (x\odot_{h^{*}(V)}y)\odot_{h^{*}(V)}z = ((x_{1},
y_{1},a_{1})\odot_{h^{*}(V)}(y_{1},
y_{2},a_{2}))\odot_{h^{*}(V)}(y_{2}, y_{3},a_{3}) =$\\
$~~~~~~~= (x_{1},
y_{2},a_{1}\odot_{V}a_{2})\odot_{h^{*}(V)}(y_{2}, y_{3},a_{3})=
(x_{1}, y_{3},(a_{1}\odot_{V}a_{2})\odot_{V}a_{3}))~$ and\\[0.2cm]
$(2)~~~ x\odot_{h^{*}(V)} (y\odot_{h^{*}(V)}z) = (x_{1},
y_{1},a_{1})\odot_{h^{*}(V)}((y_{1},
y_{2},a_{2})\odot_{h^{*}(V)}(y_{2}, y_{3},a_{3}))=$\\
$~~~~~~~= (x_{1}, y_{1},a_{1})\odot_{h^{*}(V)}( y_{1},
y_{3},a_{2}\odot_{V}a_{3})= (x_{1},
y_{3},a_{1}\odot_{V}(a_{2}\odot_{V}a_{3}))$.

Using the relations (1),(2) and the fact that $\odot_{V}$ is
associative in $V$, it follows $~
(x\odot_{h^{*}(V)}y)\odot_{h^{*}(V)}z =x\odot_{h^{*}(V)}
(y\odot_{h^{*}(V)}z).$ Hence, the condition (G1) from Definition
2.1 holds.

$(b)~$ We have\\[0.1cm]
$\varepsilon^{*}(\alpha^{*}(x))\odot_{h^{*}(V)} x=
\varepsilon^{*}(\alpha^{*}(x_{1}, y_{1},a_{1}))\odot_{h^{*}(V)}
(x_{1}, y_{1},a_{1})=\varepsilon^{*}(x_{1})\odot_{h^{*}(V)}
(x_{1}, y_{1},a_{1})=$\\[0.1cm]
$~~~~~= (x_{1}, x_{1},\varepsilon(h(a_{1}))) \odot_{h^{*}(V)}
(x_{1}, y_{1},a_{1})= (x_{1}, y_{1},
\varepsilon(h(a_{1}))\odot_{V}a_{1})=(x_{1},
y_{1},a_{1})=x$.\\[-0.2cm]

Similarly, we verify that $x\odot_{h^{*}(V)}
\varepsilon^{*}(\beta^{*}(x))=x.$ Hence, the condition (G2) from
Definition 2.1 holds.

$(c)~$ We have\\[0.1cm]
$(3)~~~i^{*}(x)\odot_{h^{*}(V)} x= i^{*}(x_{1},
y_{1},a_{1})\odot_{h^{*}(V)}(x_{1}, y_{1},a_{1})=(y_{1},
x_{1},i(a_{1}))\odot_{h^{*}(V)}(x_{1}, y_{1},a_{1})=$\\[0.1cm]
$~~~~~=(y_{1}, y_{1},i(a_{1})\odot_{V} a_{1})=
(y_{1},y_{1},\varepsilon(\beta(a_{1}))) $ and\\[0.1cm]
$(4)~~~\varepsilon^{*}(\beta^{*}(x))=
\varepsilon^{*}(\beta^{*}(x_{1}, y_{1},a_{1}))=
\varepsilon^{*}(y_{1})=(y_{1},y_{1},\varepsilon(h(y_{1})))=(y_{1},
y_{1},\varepsilon(\beta(a_{1}))) $.\\[-0.2cm]

From the relations (3) and (4) it follows $
i^{*}(x)\odot_{h^{*}(V)} x= \varepsilon^{*}(\beta^{*}(x))$.

Similarly, we prove that $ x\odot_{h^{*}(V)} i^{*}(x)=
\varepsilon^{*}(\alpha^{*}(x))$. Hence, the condition (G3) from
Definition 2.1 holds. Then $(h^{*}(V), \alpha^{*}, \beta^{*},
\odot_{h^{*}(V)},\varepsilon^{*}, i^{*}, X )$ is
 a groupoid.

 $(ii)~ $ We prove that the conditions from Definition 3.1 are verified.

 It is clearly that the condition (3.1.1) from Definition 3.1 is
 verified.

Let now two elements $ x,y\in h^{*}(V)$ and $ k_{1}, k_{2} \in K$,
 where $ x = (x_{1}, y_{1}, a_{1}) $ and $ y = (x_{2}, y_{2}, a_{2})
 $. We have\\[0.1cm]
$ \alpha^{*}(k_{1} x + k_{2} y)= \alpha^{*}( k_{1}x_{1}+ k_{2}
x_{2}, k_{1}
y_{1}+ k_{2} y_{2}, k_{1} a_{1}+  k_{2} a_{2})=$\\
$=  k_{1} x_{1}+ k_{2} x_{2} = k_{1} \alpha^{*}(x_{1}, y_{1},
a_{1}) + k_{2} \alpha^{*}(x_{2}, y_{2}, a_{2}) = k_{1}
\alpha^{*}(x) +k_{2} \alpha^{*}(y).$\\[-0.3cm]

 It follows that $ \alpha^{*}$ is a linear map. Similarly we prove that $ \beta^{*}$ is a linear map.
Therefore the conditions (3.1.2) from Definition 3.1 hold.

Let now two elements $ x_{1}, x_{2}\in X$ and $ k_{1}, k_{2} \in
K$. Using the fact that $h$ is a linear map, one obtains that $ \varepsilon^{*}$ is linear. Indeed, we have\\[0.1cm]
$\varepsilon^{*}(k_{1} x_{1} + k_{2} x_{2})= ( k_{1}x_{1}+ k_{2}
x_{2}, k_{1} x_{1}+ k_{2} x_{2}, h(k_{1} x_{1}+  k_{2} x_{2})=$\\
$= ( k_{1} x_{1}+ k_{2} x_{2},k_{1} x_{1}+ k_{2} x_{2},k_{1}
h(x_{1})+ k_{2} h(x_{2}))   = k_{1} (x_{1}, x_{1}, h(x_{1})) +
k_{2}(x_{2}, x_{2}, h(x_{2})) =$\\
$= k_{1} \varepsilon^{*}(x_{1}) + k_{2} \varepsilon^{*}(x_{2}).$\\[-0.3cm]

Let $ x,y\in h^{*}(V)$ and $ k_{1}, k_{2} \in K$,
 where $ x = (x_{1}, y_{1}, a_{1})$ and $ y = (x_{2}, y_{2}, a_{2})
 $. Using the linearity of the map $ i$, one obtains that $i^{*}$ is linear. Indeed, we have\\[0.1cm]
$ i^{*}(k_{1} x + k_{2} y)= i^{*}( k_{1}x_{1}+ k_{2} x_{2}, k_{1}
y_{1}+ k_{2} y_{2}, k_{1} a_{1}+  k_{2} a_{2})=$\\
$= ( k_{1} y_{1}+ k_{2}y_{2}, k_{1} x_{1}+ k_{2} x_{2}, i(k_{1}
a_{1}+ k_{2} a_{2})) = ( k_{1} y_{1}+ k_{2}y_{2}, k_{1}
x_{1}+k_{2} x_{2}, k_{1}i(a_{1})+ $\\
$+ k_{2}i(a_{2}))= k_{1}( y_{1}, x_{1}, i(a_{1}))+ k_{2}( y_{2},
x_{2}, i(a_{2})) = k_{1} i^{*}(x_{1},
y_{1}, a_{1}) + k_{2} i^{*}(x_{2}, y_{2}, a_{2}) =$\\
$= k_{1} i^{*}(x)+ k_{2} i^{*}(y).$\\[-0.3cm]

 Let now $(x,y,a)\in h^{*}(V)$. Then $ h(x)=\alpha(a) $ and $ h(y)=\beta(a) $. We have\\[0.2cm]
$(x,y,a) + i^{*}(x,y,a) = (x,y, a) + (y,x,i(a))= (x+y, y+x,
a+i(a)),~$ and\\[0.2cm]
$ \varepsilon^{*}(\alpha^{*}(x,y,a))+
\varepsilon^{*}(\beta^{*}(x,y,a))
=\varepsilon^{*}(x)+\varepsilon^{*}(y)=(x, x, \varepsilon(h(x))) +
(y, y, \varepsilon(h(y))) =$\\
$= (x+y, x+y, \varepsilon(h(x))+\varepsilon(h(y)))=(x+y, x+y,
\varepsilon(\alpha(a))+\varepsilon(\beta(a)))=$\\
$=(x+y, y+x, a+i(a)).$\\[-0.3cm]

It follows $~(x,y,a) + i^{*}(x,y,a) =
\varepsilon^{*}(\alpha^{*}(x,y,a))+
\varepsilon^{*}(\beta^{*}(x,y,a))$, and so the condition
(3.1.3)(1) from Definition 3.1 holds.

For to verify the relation $ 3.1.4 (1)$ from Definition 3.1 we
consider the arbitrary elements  $ x, y, z \in h^{*}(V)$, where $
x = (x_{1}, y_{1}, a_{1}), y =(x_{2}, y_{2}, a_{2}) $ and $ z =
(x_{3}, y_{3}, a_{3}) $  such that $ \alpha^{*}(y) = \beta^{*}(x)
=\alpha^{*}(z) $.  Then $ x_{2} = y_{1} = x_{3} $ and follows $ x
= (x_{1}, y_{1}, a_{1} ), y = ( y_{1}, y_{2}, a_{2}) $ and $ z =
(y_{1}, y_{3}, a_{3}) $ with $a_{1}, a_{2}, a_{3} \in V $ such
that $ h(x_{j})=\alpha(a_{j})$ and $ h(y_{j})=\beta(a_{j})$ for $
j=1,2,3.$\\[0.2cm]
$(5)~~~ x\odot_{h^{*}(V)}( y + z - \varepsilon^{*}(\beta^{*}(x))) = (x_{1}, y_{1}, a_{1} )\odot_{h^{*}(V)}( ( y_{1}, y_{2}, a_{2}) +$\\
 $+ (y_{1}, y_{3}, a_{3}) - \varepsilon^{*}(\beta^{*}(x_{1}, y_{1}, a_{1})))
 = (x_{1}, y_{1}, a_{1} )\odot_{h^{*}(V)}( ( y_{1}, y_{2}, a_{2})
 +(y_{1}, y_{3}, a_{3}) - \varepsilon^{*}(y_{1}))= (x_{1}, y_{1}, a_{1} )\odot_{h^{*}(V)}( ( y_{1}, y_{2}, a_{2})
 +(y_{1}, y_{3}, a_{3}) -( y_{1}, y_{1},\varepsilon(h(y_{1})))=$\\
 $=(x_{1}, y_{1}, a_{1}
 )\odot_{h^{*}(V)}(y_{1},y_{2}+y_{3}-y_{1},
 a_{2}+a_{3}-\varepsilon(\beta(a_{1})))=$\\
$= (x_{1},y_{2}+y_{3}-y_{1}, a_{1}
\odot_{V}(a_{2}+a_{3}-\varepsilon(\beta(a_{1})))~$ and\\[0.2cm]
$(6)~~~ x\odot_{h^{*}(V)} y + x\odot_{h^{*}(V)} z - x = (x_{1},
y_{1}, a_{1} )\odot_{h^{*}(V)}(  y_{1}, y_{2}, a_{2}) + $\\
$+ (x_{1}, y_{1}, a_{1} )\odot_{h^{*}(V)}(  y_{1}, y_{3}, a_{3}) -
(x_{1}, y_{1}, a_{1} )= ( x_{1}, y_{2}, a_{1}\odot_{V} a_{2} ) + $\\
$+ ( x_{1}, y_{3}, a_{1}\odot_{V} a_{2} ) -(x_{1}, y_{1}, a_{1} )
= ( x_{1}, y_{2}+ y_{3} - y_{1}, a_{1}\odot_{V} a_{2} +
a_{1}\odot_{V} a_{3}- a_{1} ).$\\[-0.3cm]

Using (5), (6) and the hypothesis\\
 $ a_{1}\odot_{V}(a_{2}+a_{3}-\varepsilon(\beta(a_{1})))= a_{1}\odot_{V}
a_{2} + a_{1}\odot_{V} a_{3}- a_{1} $,  one obtains\\
 $ x\odot_{h^{*}(V)}( y + z - \varepsilon^{*}(\beta^{*}(x))) = x\odot_{h^{*}(V)} y + x\odot_{h^{*}(V)} z - x$.

 Hence the condition 3.1.4(1) from Definition 3.1 holds.

Let now $ x = (x_{1}, y_{1}, a_{1} )\in h^{*}(V), y = ( y_{1},
y_{2}, a_{2})\in h^{*}(V) $ such that $ h(x_{1})=\alpha(a_{1}),
h(y_{1})=\alpha(a_{2})= \beta(a_{1}) $ and $ h(y_{2})=
\beta(a_{2}).$  For all $ k\in K$, we have\\[0.2cm]
$(7)~~ x \odot_{h^{*}(V)}( k y + (1-k)\varepsilon^{*}(\beta^{*}(x)
)) = (x_{1}, y_{1}, a_{1} )\odot_{h^{*}(V)} ( k (y_{1}, y_{2},
a_{2})+$\\
$+(1-k)\varepsilon^{*}(\beta^{*}(x_{1}, y_{1}, a_{1}) ) )= (x_{1},
y_{1}, a_{1} )\odot_{h^{*}(V)} ( k (y_{1}, y_{2}, a_{2}) +
(1-k)\varepsilon^{*}(y_{1}))=$\\
$= (x_{1}, y_{1}, a_{1} )\odot_{h^{*}(V)} ( k (y_{1}, y_{2},
a_{2}) + (1-k)\varepsilon(h(y_{1}) ) )=
 (x_{1}, y_{1}, a_{1}
)\odot_{h^{*}(V)} ( (ky_{1}, ky_{2}, ka_{2}) +$\\
$+ (1-k)( y_{1}, y_{1}, \varepsilon(\beta(a_{1}) ) )= (x_{1},
y_{1}, a_{1} )\odot_{h^{*}(V)} (y_{1},ky_{2}+ (1-k) y_{1}, ka_{2}+
\varepsilon(\beta(a_{1}) ) )=$\\
$= (x_{1},ky_{2}+ (1-k) y_{1}, a_{1}\odot_{V}(ka_{2}+
(1-k)\varepsilon(\beta(a_{1}) ))~ $
and\\[0.2cm]
$(8)~~ k ( x \odot_{h^{*}(V)} y ) + (1-k) x = k ((x_{1}, y_{1},
a_{1} )\odot_{h^{*}(V)}(y_{1}, y_{2}, a_{2}) ) +$\\
$+ (1-k)(x_{1}, y_{1}, a_{1} )= k (x_{1}, y_{2},
a_{1}\odot_{V}a_{2})+ (1-k)(x_{1}, y_{1}, a_{1} ) =$\\
$= ( x_{1}, ky_{2}+ (1- k) y_{1}, k (a_{1}\odot_{V}a_{2})+
(1-k)a_{1}).~$\\[-0.3cm]

Using the equalities (7) and (8) and the hypothesis\\
 $a_{1}\odot_{V}(ka_{2}+ (1-k)\varepsilon(\beta(a_{1})) )= k
(a_{1}\odot_{V}a_{2})+ (1-k)a_{1} $,\\
 one obtains that the condition 3.1.4(2) from Definition 3.1 holds.

In the same manner we prove that the conditions 3.1.4 (3) and
3.1.4 (4) hold. Hence $(h^{*}(V), \alpha^{*}, \beta^{*},
\odot_{h^{*}(V)},\varepsilon^{*}, i^{*}, X )$ is a vector groupoid
\end{proof}
\begin{Prop}
The pair  $(h_{V}^{*},h):(h^{*}(V), X)\longrightarrow (V,V_{0})$
is a vector groupoid morphism, called the {\bf canonical vector
groupoid morphism on $h^{*}(V)$},  where the map
$~h_{V}^{*}:h^{*}(V)\longrightarrow V$ is defined
by:\\[-0.3cm]
$$h_{V}^{*}(x,y,a):=a,~~(\forall)~ (x,y,a)\in
h^{*}(V).$$\\[-1.2cm]
\end{Prop}
\begin{proof}
Let $(x,y,a)\in h^{*}(V).$ Then $ h(x)=\alpha(a)$ and
$h(y)=\beta(a).$ We have\\[0.2cm]
$(\alpha\circ h_{V}^{*})(x,y,a)=\alpha(
h_{V}^{*}(x,y,a))=\alpha(a)=h(x)=h(\alpha^{*}(x,y,a))=(h\circ
\alpha^{*})(x,y,a).$

Therefore, $\alpha\circ h_{V}^{*}=h\circ \alpha^{*}.$ Similarly,
we prove that $\beta\circ h_{V}^{*}=h\circ \beta^{*}.$

For all $(x_{1}, y_{1}, a_{1}), (y_{1}, y_{2}, a_{2})\in
h^{*}(V)$, we have\\[0.2cm]
$h_{V}^{*}((x_{1}, y_{1}, a_{1})\odot_{h^{*}(V)}(y_{1}, y_{2},
a_{2}))=h_{V}^{*}(x_{1}, y_{2},
a_{1}\odot_{V}a_{2})=a_{1}\odot_{V}a_{2}=$\\
$~~=h_{V}^{*}(x_{1}, y_{1},
a_{1})\odot_{V} h_{V}^{*}(y_{1}, y_{2}, a_{2}).$\\[-0.3cm]

It is easy to prove that the map $ h_{V}^{*}$ is linear. Hence,
the conditions from Definition 3.2 are verified. Therefore
$(h_{V}^{*},h)$ is a morphism of vector groupoids.
\end{proof}
\begin{Th}
 The canonical vector groupoid morphism  $(h_{V}^{*}, h):(h^{*}(V),X)\longrightarrow
(V,V_{0})$ verify the universal property:\\[0.3cm]
$({\cal P U}):$ for every vector groupoid
$(V^{\prime},\alpha^{\prime},\beta^{\prime},\odot_{V^{\prime}},
\varepsilon^{\prime},i^{\prime}, X)$ and for every vector groupoid
morphism $(u,h):(V^{\prime},X)\longrightarrow (V,V_{0})$ there
exists an unique $X$- morphism of vector groupoids
$v:V^{\prime}\longrightarrow h^{*}(V)$ such that
 the diagram:\\
$$~~ V^{\prime}~\stackrel{u}{\longrightarrow}~~V$$\\[-1cm]
$$(\exists)v {\searrow}~~~{\nearrow}h_{V}^{*}$$\\[-0.9cm]
$$~~~h^{*}(V)$$\\[-0.5cm]
is comutative, i.e. $ h_{V}^{*} \circ v=u$.
\end{Th}
\begin{proof}
 Define $ v:V^{\prime} \to  h^{*}(G)~$ by
$~v(a^{\prime}):= (\alpha^{\prime}(a^{\prime}),
\beta^{\prime}(a^{\prime}),u(a^{\prime})),~(\forall)~
a^{\prime}\in V^{\prime}.~$

We have $ v(a^{\prime})\in h^{*}(V),$ since $
h(\alpha^{\prime}(a^{\prime}))=\alpha(u(a^{\prime})) $ and $
h(\beta^{\prime}(a^{\prime}))=\beta(u(a^{\prime})) $.

We have $\alpha^{*}\circ v=\alpha^{\prime},$ since
$(\alpha^{*}\circ
v)(a^{\prime})=\alpha^{*}(v(a^{\prime}))=\alpha^{*}(\alpha^{\prime}(a^{\prime}),
\beta^{\prime}(a^{\prime}),u(a^{\prime}))=\alpha^{\prime}(a^{\prime}).$
Similarly, we verify that $ \beta^{*} \circ v=\beta^{\prime}$.

Let $(a^{\prime}, b^{\prime})\in V_{(2)}^{\prime}.$ Then $
\beta^{\prime}(a^{\prime})=\alpha^{\prime}(b^{\prime}).$ We
have\\[0.2cm]
$v(a^{\prime}\odot_{V^{\prime}}b^{\prime})=(\alpha^{\prime}(a^{\prime}\odot_{V^{\prime}}b^{\prime}),
\beta^{\prime}(a^{\prime}\odot_{V^{\prime}}b^{\prime}),u(a^{\prime}\odot_{V^{\prime}}b^{\prime}))=
(\alpha^{\prime}(a^{\prime}),
\beta^{\prime}(b^{\prime}),u(a^{\prime})\odot_{V}u(b^{\prime}))$
and\\
$v(a^{\prime})\odot_{h^{*}(V)}v(b^{\prime})=(\alpha^{\prime}(a^{\prime}),
\beta^{\prime}(a^{\prime}),u(a^{\prime}))\odot_{h^{*}(V)}(\beta^{\prime}(a^{\prime}),
\beta^{\prime}(b^{\prime}),u(b^{\prime}))=$\\
$=(\alpha^{\prime}(a^{\prime}),
\beta^{\prime}(b^{\prime}),u(a^{\prime})\odot_{V}u(b^{\prime}))$.

It follows that $v(a^{\prime}\odot_{V^{\prime}}b^{\prime})=
v(a^{\prime})\odot_{h^{*}(V)}v(b^{\prime}).$

Using the linearity of the maps $ \alpha^{\prime}, \beta^{\prime}$
and $u$, it is easy to verify that $v$ is a linear map. Therefore,
$v$ is a  $~X-$ morphism of vector groupoids.

We have $ ( h_{V}^{*} \circ
v)(a^{\prime})=h_{V}^{*}(v(a^{\prime}))=h_{V}^{*}(\alpha^{\prime}(a^{\prime}),
\beta^{\prime}(a^{\prime}),u(a^{\prime}))= u(a^{\prime}), $ for
all $a^{\prime}\in V^{\prime}$, i.e. $ h_{V}^{*} \circ v =u.$
 Finally,  we prove by a standard manner that $v$ is unique.
\end{proof}
\begin{Prop}
The induced vector groupoid  $ (h^{*}(V), X)$ of a transitive vector
groupoid $ (V, V_{0})$ via the linear map
 $h:X\longrightarrow V_{0}$ is transitive.
 \end{Prop}
 \begin{proof}
 We prove that the anchor $ ( \alpha^{*}, \beta^{*}): h^{*}(V) \to X\times X$ is a surjective map.
For this, let an arbitrary element $(x,y)\in X\times X.$ We have
$(h(x),h(y))\in V_{0}\times V_{0}.$ Since the anchor $(\alpha, \beta
):V\to V_{0}\times V_{0}$ of the groupoid $V$ is surjective, there
exists an element  $a\in V$ such that $(\alpha ,\beta)(a)= (h(x),
h(y)). $ Then $\alpha(a)= h(x), \beta(a)=h(y)$ and $(x,y,a)\in
h^{*}(V).$  We have that $((\alpha^{*}, \beta^{*})(x,y,a)=
(\alpha^{*}(x,y,a), \beta^{*}(x,y,a))=(x,y)$. Consequently,
$(\alpha^{*}, \beta^{*})$ is surjective. Hence $h^{*}(V)$ is a
transitive vector groupoid.
\end{proof}

\vspace*{0.4cm}

\hspace*{0.7cm}West University of Timi\c soara\\
\hspace*{0.7cm} Department of Mathematics\\
\hspace*{0.7cm} Bd. V. P{\^a}rvan,no.4, 300223, Timi\c soara, Romania\\
\hspace*{0.7cm}E-mail:vpoputa@yahoo.com; ivan@math.uvt.ro

\end{document}